\documentclass[11pt]{article}
\usepackage{amssymb}
\usepackage{amsmath}
\usepackage{amsfonts}
\usepackage{amsthm}
\newtheorem{theorem}{Theorem}
\newtheorem{example}[theorem]{Example}

\begin{document}
\begin{center}{\Large Varieties in $(\mathbf{P}^1(\overline{\mathbf{F}}))^n$\\
    by Elimination and Extension}\end{center}
\begin{center}{\large Douglas A. Leonard\\    
    Department of Mathematics and Statistics\\
    Auburn University}\end{center}

\begin{abstract}
  This paper contains a theory of elimination and extension to compute varieties symbolically,
  based on using {\em coordinates} from $(\mathbf{P}^1(\overline{\mathbf{F}}))^n$
  and disjoint {\em parts} of varieties (defined by both equality and inequality constraints),
  leading to a recursive algorithm to compute said varieties
  by extension at the level of {\em parts} of a variety.
  {\sc Macaulay2} code for this is included along with an example. 
  This is a first step in the author's project of giving a purely algebraic
  theory of desingularization of function fields,
  in that that project relies heavily on using this type of coordinates for function field elements
  and on partitioning a set of valuations into disjoint sets similarly.
\end{abstract}  
\newpage

\section{Introduction}

Given an ordered set of coordinate functions $(x_n,\ldots,x_1)$,
and an ideal $I:=I(x_n,\ldots,x_1)\subseteq \overline{\mathbf{F}}[x_n,\ldots,x_1]$
of all the polynomial relations among them,
it is of interest to consider the variety
\[V(I):=\{ (\overline{x}_n,\ldots,\overline{x}_1)\in T^n\ :\
  b(\overline{x}_n,\ldots,\overline{x}_1)=0\mbox{ for all }b\in I\}\]
For any affine coordinate functions used here $T=\overline{\mathbf{F}}$, an algebraically closed field,
but for rational coordinate functions (elements of a function field),
$T=\mathbf{P}^1(\overline{\mathbf{F}})$, the projective line over that algebraically closed field,
in that for $x_j=g_j/h_j$, it is natural to expect
$\overline{g_j/h_j}\in T$ to be the inverse of $\overline{h_j/g_j}\in T$,
even in the case that one is $0/1$ and the other $1/0$.
We will primarily be dealing with coordinate values from
$(\mathbf{P}^1(\overline{\mathbf{F}}))^n$ in this paper,
though we will embed this problem into an affine problem with coordinates in
$\overline{\mathbf{F}}^{2n}$ to do the extension.

The philosophy behind determining {\em all} the elements of a variety
by elimination and extension is to work one coordinate at a time,
finding all possibilities for coordinate $\overline{x}_1\in T$ first,
and then recursively finding all possibilities
for coordinates $(\overline{x}_{j+1},\ldots,\overline{x}_1)\in T^{j+1}$
given $(\overline{x}_j,\ldots,\overline{x}_1)\in T^j$.
This is analagous to row-reduction and back-substitution in linear algebra.
It should be expected to produce exactly the elements of the variety,
and it should produce the same set of elements
for any choice (out of $n!$ possibilities) of the ordering of the variables in the lex monomial ordering used.
Extension doesn't really work at the level of varieties,
but rather at the level of disjoint parts $S$ of a partition of the variety,
each part defined by a (finite) set of polynomial equality constraints $EQ(S)$
and a (finite, possibly empty) set of inequality constraints $NEQ(S)$.

Such partitions are crucial in doing desingularization of function fields as well, \cite{Leon}.
For instance, the Whitney umbrella, Example 3.6.1 in \cite{Ko}:
\[V:=\{(\overline{x}_3,\overline{x}_2,\overline{x}_1)\in\overline{\mathbf{F}}^3\ :\ \overline{x}_3\overline{x}_2^2-\overline{x}_1^2=0\}\]
is {\em singular} along the line
\[L:=\{(\overline{x}_3,\overline{x}_2,\overline{x}_1)\in\overline{\mathbf{F}}^3\ :\ \overline{x}_2=0=\overline{x}_1\},\]
but has a more complicated {\em singularity} at the point $P$ with $\overline{x}_3=\overline{x}_2=\overline{x}_1=0$.
The discussion ensuing in \cite{Ko} is then in terms of whether to {\em blow up} the variety, $L$ or the variety $P$,
rather than dealing with the disjoint parts $P$, $L\backslash P$, and even the part $L^c$ consisting of the non-singular points.

So we'll start with notation to describe what elimination and extension should look like in general,
then consider how to deal with this relative to partitioning the variety.
The actual theorem and its proof are relatively short, just explaining how $(\overline{x}_j,\ldots,\overline{x}_1)\in T^j$
satisfying the constraints of a part $S$ extend to $(\overline{x}_{j+1},\ldots,\overline{x}_1)\in T^{j+1}$
satisfying the constraints of a part $S^*$.
Even the {\sc Macaulay2} code given to implement this is not very long by code standards.

\section{Notation for elimination and extension}
Let $\overline{\mathbf{F}}$ be an algebraically closed field
(here for computational reasons with $\mathbf{F}$ restricted to being
the rationals, $\mathbf{Q}$, in characteristic $0$
or the finite field of $p$ elements, $\mathbf{F}_p$, in characteristic $p>0$).
Let $R:=\overline{\mathbf{F}}[x_n,\ldots,x_1]$
with {\em lex} $x_n\succ\cdots\succ x_1$ monomial ordering
(an example of an {\em elimination order}).

Let $B$ be a minimal, reduced (hence finite) (lex) Gr\"obner basis
for the ideal $I$ of $R$ that it generates.

Define
\[R_j:=\overline{\mathbf{F}}[x_j,\ldots,x_1];\]
\[I_j:=I\cap R_j;\]
\[B_j:=B\cap R_j;\]
\[V_j:=\{(\overline{x}_j,\ldots,\overline{x}_1)\in T^j\ :\
  b(\overline{x}_j,\ldots,\overline{x}_1)=0\mbox{ for all }b\in B_j\}\]
What $T$ is is a central point of this paper.
Then the general form of elimination and extension would be roughly as follows.

\begin{theorem}[Elimination]
  ${\ }$
  
  \begin{enumerate}
    
\item $I_j$ is an $($elimination$)$ ideal of $R_j$, $1\leq j\leq n$.
\item $B_j$ is a $($lex$)$ Gr\"obner basis for $I_j$, $1\leq j\leq n$.
\item $\{ (\overline{x}_j,\ldots,\overline{x}_1)\in T^j\ :\  (\overline{x}_n,\ldots,\overline{x}_1)\in V_n \}\subseteq V_j$, $1\leq j\leq n$.
\end{enumerate}
\end{theorem}

The proof should be a straight-forward exercise.
[A proof given for the affine case in \cite{CLO} is rather short,
but the advantage of the reader trying this is to see
where the lex ordering is used
and in trying to understand that the third item is not always an equality,
though it will be for the coordinates used here.]  

\begin{theorem}[Extension]
  ${\ }$
  
  \begin{enumerate}

  \item If $(\overline{x}_j,\ldots,\overline{x}_1)\in V_j\subseteq T^j$,
  then there is at least one $\overline{x}_{j+1}\in T$
  such that $(\overline{x}_{j+1},\ldots,\overline{x}_1)\in V_{j+1}\subseteq T^{j+1}$.
\item All such $\overline{x}_{j+1}$ can be computed symbolically.
\item $V_j=\{ (\overline{x}_j,\ldots,\overline{x}_1)\in T^j\ :\  (\overline{x}_n,\ldots,\overline{x}_1)\in V_n\}$.
\end{enumerate}
\end{theorem}

The proof of extension is another matter altogether,
in that this is not always the case for affine varieties
(meaning $V_n\subseteq \overline{\mathbf{F}}^n$).

The simple example $B:=(x_2x_1-1)=B_2$, $B_1=\emptyset$
has $V_1=\overline{\mathbf{F}}$,
and $(0)\in V_1$ does not extend to $(\overline{x}_2,0)\in V_2$
since  $\overline{x}_2\cdot 0-1=-1\neq 0$.
[Of course, it is the claim here that
$((0:1))\in(\mathbf{P}^1(\overline{\mathbf{F}}))^1$ extends to
$((1:0),(0:1))\in(\mathbf{P}^1(\overline{\mathbf{F}}))^2$.]

Another such simple example $B:=(x_2x_1)=B_2$, $B_1=\emptyset$
has $V_1=\overline{\mathbf{F}}$,
and $(0)\in V_1$ should extend to $(\overline{x}_2,0)\in V_2$
for any  $\overline{x}_2\in\overline{\mathbf{F}}$
(with $\overline{x}_1\neq 0$ extending to $(0,\overline{x}_1)\in V_2$).
But a theorem such as \cite{CLO} [Theorem 3.1.3]
that tries to deal with this example
by trying only to extend if $\overline{x}_1\neq 0$, would miss the former case.

So here varieties will be subsets of $(\mathbf{P}^1(\overline{\mathbf{F}}))^n$.
Then such varieties will be partitioned into (disjoint) {\em parts}, with part $S$,
defined by a finite set $EQ(S)$ of equality constraints on the coordinates $((g_n:h_n),\ldots,(g_1:h_1))$
and a finite (possibly empty) set $NEQ(S)$ of inequality constraints as well
(again as opposed to having varieties $V$ only defined by equality constraints given by $I(V)$).
$EQ(S)$ will include the non-homogeneous equality constraints 
$h_i(h_i-1),\ (g_i-1)(h_i-1)$ for each $1\leq i\leq n$
that force a canonical representative $(1:0)$ or $(\overline{g}_i:1)$
for each point of the projective line.

The only other ingredients will be a mapping 
\[ \phi\ :\ \overline{\mathbf{F}}[g_n,h_n,\ldots,g_1,h_1]\to
  \overline{\mathbf{F}}[y_{2n},y_{2n-1},\ldots,y_2,y_1]\]
to blur the distinction between the $g_j$'s and the $h_j$'s in doing extension;
and the further mappings
\[ \phi_j\ :\ \overline{\mathbf{F}}[y_{2n},\ldots,y_1]\to
              \overline{\mathbf{F}}[\overline{y}_j,\ldots \overline{y}_1][y_{2n},\ldots,y_{j+1}]\]
used to identify leading coeficients $lc(f)\in  \overline{\mathbf{F}}[\overline{y}_j,\ldots, \overline{y}_1]$
that lead to different extensions depending on whether $lc(f)$ can be $0$ or not.
Some parts $S$ will then be partitioned into two (disjoint) parts
by appending the constraint $lc(f)$ to $EQ(S)$ or $NEQ(S)$,
based on whether such leading coefficient takes on the value $0$ or not,
if $lc(f)$ is not already known to be non-zero.
(This leads to computing a (finite) Gr\"obner basis for
either $\langle EQ(S)\rangle +\langle lc(f)\rangle$
or $saturation(\langle EQ(S)\rangle,\langle lc(f)\rangle$ respectively to get the new equality constraints,
and/or appending $lc(f)$ to $NEQ(S)$ in the latter to get the new inequality constraints.)

So, given a variety $V=\mathbf{V}(I)$ for $I$ an ideal of $\overline{\mathbf{F}}[x_n,\ldots,x_1]$,
first replace each $x_j$ by $g_j/h_j$ to symbolically view $x_j$ as a rational function.
Then turn the generator polynomials $b$ of $I$ into polynomials:
\[ b^*(g_n,h_n,\ldots,g_1,h_1):=\left(\prod_{j=1}^nh_j^{deg(b,x_j)}\right)b(g_n/h_n,\ldots,g_1/h_1)\]
that are homogeneous in each pair $(g_k,h_k)$, $1\leq k\leq n$.
Use the map
\[ \phi\ :\ \overline{\mathbf{F}}[g_n,h_n,\ldots,g_1,h_1]\to \overline{\mathbf{F}}[y_{2n},\ldots,y_1]\]
defined by $\phi(g_j):=y_{2j}$ and $\phi(h_j):=y_{2j-1}$ for $1\leq j\leq n$.
Append the non-homogeneous equality constraints $y_{2j-1}(y_{2j-1}-1)=0$, and
$(y_{2j}-1)(y_{2j-1}-1)=0$ for $1\leq j\leq n$ to force a canonical choice for
representatives of the elements of the projective line as either $(1:0)$ or $(\overline{y_{k}}:1)$ for $1\leq k\leq 2n$.

Consider the further maps
\[ \phi_j\ :\ \overline{\mathbf{F}}[y_{2n},\ldots,y_1]\to
  \overline{\mathbf{F}}[\overline{y}_{j},\ldots,\overline{y}_1][y_{2n},\ldots,y_{j+1}]\]
defined by $\phi_j(y_k):=\overline{y}_k$ for $k\leq j$ and $\phi_j(y_k):=y_k$ for $k> j$.

Computations will be done symbolically in these subrings
\[R_j:=\overline{\mathbf{F}}[\overline{y}_{j},\ldots,\overline{y}_1][y_{2n},\ldots,y_{j+1}]\]
though ultimately any $(\overline{y}_{2n},\ldots,\overline{y}_1)\in\overline{\mathbf{F}}^{2n}$
will have to be reinterpreted as an element of $(\mathbf{P}^1(\overline{\mathbf{F}}))^n$
by viewing each $(y_{2k},y_{2k-1})\in\overline{\mathbf{F}}^2$
as $(y_{2k}:y_{2k-1})\in\mathbf{P}^1(\overline{\mathbf{F}})$ for $1\leq k\leq n$.

[Actually, computationally we can get away with using only the ring
\[R:=\overline{\mathbf{F}}[z_{2n},\ldots,z_1][y_{2n},\ldots,y_1]\]
so as to cut down on the number of rings and ring maps needed.]

\newpage

\section{Theorem}
\begin{theorem}[The Extension Theorem for coordinates in $(\mathbf{P}^1(\overline{\mathbf{F}}))^n$]
  
Given the preceding setup, suppose that for some part $S$,
\[ S|_{R_j}:=\{(\overline{y}_j,\ldots,\overline{y}_1)\in\overline{\mathbf{F}}^j\ :\]
\[  b(\overline{y}_j,\ldots,\overline{y}_1)=0,\mbox{ for all } b\in(EQ(S)\cap R_j)\]
\[\mbox{ and }
  b(\overline{y}_j,\ldots,\overline{y}_1)\neq 0,\mbox{ for all } b\in(NEQ(S)\cap R_j)\}\]
is known, and is to be extended to one or more parts of the form

\[ S^*|_{R_{j+1}}:=\{(\overline{y}_{j+1},\ldots,\overline{y}_1)\in\overline{\mathbf{F}}^{j+1}\ :\]
\[
  b(\overline{y}_{j+1},\ldots,\overline{y}_1)=0,\mbox{ for all } b\in(EQ(S^*)\cap R_{j+1})\]
\[\mbox{ and }
  b(\overline{y}_{j+1},\ldots,\overline{y}_1)\neq 0,\mbox{ for all } b\in(NEQ(S^*)\cap R_{j+1})\}\]
by finding polynomial restrictions on the choice of $\overline{y}_{j+1}$ for each such part $S^*$.

This can be done as follows:
\begin{enumerate}
  \item Consider those $b_i(y_{j+1})\in (\phi_j(EQ(S))\cap\phi_j(R_{j+1}))\backslash \phi_j(R_j)$
      in increasing lex monomial order, with $d_i:=degree(b_i,y_{j+1})$.
  \item Let $lc_i:=LC(b_i(\overline{y}_j,\ldots,\overline{y}_1))\in \phi_j(R_j)$.
  \item If $lc_1$ could take on a non-zero or a zero value, then $S$ needs to be partitioned
      into two $($disjoint$)$ parts relative to $lc_1$ being non-zero or not before proceeding.
      But assuming that $lc_1$ can only take on non-zero values,
      either because it is explicitly a non-zero field element
      or because it is a factor of an element in $NEQ(S)$,
      choose $\overline{y}_{j+1}$ to be a $($symbolic$)$ root of $b_1(y_{j+1})$
      $($even if the explicit roots could be computed$)$.
\end{enumerate}
\end{theorem}

\begin{proof} Suppose there were some $b_s(y_{j+1})$ for which $b_s(\overline{y}_{j+1})\neq 0$.
Assume $s$ is chosen smallest relative to this.
Then $lc_1b_s(y_{j+1})-lc_sb_1(y_{j+1})$ has degree less than $d_s$, so is reducible to $0$
using only elements of $EQ(S)$ preceding $b_s$ in the lex monomial ordering.
But all of these are $0$ at $(\overline{y}_{j+1},\ldots,\overline{y}_1)$, as is $b_1$.
So $lc_1b_s$ is $0$ as well.
But $lc_1\neq 0$, forcing  $b_s(\overline{y}_{j+1})= 0$, a contradiction.
\end{proof}

\begin{example}
  Consider the ideal $I=\langle x_1(x_3^2x_2+x_3+1),x_3(x_3^2x_2+x_3+1)\rangle$,
and its $($affine$)$ variety $V$. 
Since $B_1=B_2=\emptyset$,
$V_1=\overline{\mathbf{F}}^1$ and $V_2=\overline{\mathbf{F}}^2$. 
If $\overline{x}_2\neq 0$, then the affine extension theorem in \cite{CLO} would extend this
correctly for $\overline{x}_3\ :\ \overline{x}_3^2\overline{x}_2+\overline{x}_3+1=0$.
But it does not apply to the case  $\overline{x}_2=0$.
In this case, $(0,0)$ should extend to either $(0,0,0)$ or $(-1,0,0)$, while
$(0,\overline{x}_1)$ with $\overline{x}_1\neq 0$ can be extended to $(-1,0,\overline{x}_1)$ only.
This example is worked out using the {\sc Macaulay2} code below, with the edited result given at the end. 
\end{example}

\section{Macaulay2 code}
What follows is the author's {\sc Macaulay2} code and its application to this example
(with $z_i$ for $\overline{y}_i$, and $EQ\# i$ and $NEQ\# i$ for $EQ(S_i)$ and $NEQ(S_i)$). 
Everything happens inside the one ring $R$ to save having to map elements and ideals of one ring into
another all the time. 
The part numbered $17$ is the affine part that the affine CLO theorem 3.1.3
mentioned above doesn't deal with;
$14,16$ and half of $15$ are the other affine parts that it would deal with;
and $8,10,11,12,18$ and the other half of $15$
have at least one non-affine coordinate.

\begin{verbatim}
--A Gr\"obner basis as an ideal instead of a matrix
GB:=(I)->ideal flatten entries gens gb I
---------------------------------------------------
--symbolic LC that could be zero
redCoeff:=(LC,NEQk)->(
    if NEQk !={} then(
        ilc=ideal(promote(LC,ring(NEQk#0)));
        for i to #NEQk-1 do(
            ilc=saturate(ilc,ideal(NEQk#i));
	    );
        lc=(gens(ilc))_(0,0);
    )
    else( 
        lc=LC;
    );
    lc
)
---------------------------------------------------
rad:=(f,R)->flatten entries gens radical promote(ideal(f),R)
---------------------------------------------------
multihomRing:=(n,field)->(
--reverse ordering of subscripts-------------------
    l=for i to 2*n-1 list 2*n-i;
--subscripted variables
    ll=for i to #l-1 list y_(l#i);
--subcripted coordinate values
    lll=for i to #l-1 list z_(l#i);
--ring of subcripted coordinate values
    F=field[lll,MonomialOrder=>Lex];
--ring of subcripted variables
    R=F[ll,MonomialOrder=>Lex]
    );
---------------------------------------------------
multihomVariety:=(n,R,multihom)->(
--non-homogeneous constraints to force canonical reps for elements of P^1
    nonhom=ideal(
        for i to 2*n-1 list 
            if i%2==0 then y_(i+1)*(y_(i+1)-1) 
	    else (y_(i+1)-1)*(y_i-1));
    EQ={GB radical (multihom+nonhom)};
    NEQ={{}};
    PREV={-1};
----------------------------------------------------
    phi:=(j)->map(R,R,matrix{ 
        for i to 2*n-1 list(
            if i>= 2*n-j 
            then z_(2*n-i) 
            else y_(2*n-i)
            )}
        );
----------------------------------------------------
    psi=map(R,R,matrix{gens(R)}|matrix{gens(R)});
----------------------------------------------------
    currentnode=0; 
    nextnode=1;
    sizeEQ=1;
    while currentnode < sizeEQ do(
        varno=1;
        found=0;
        while found==0 and varno< 2*n do(
            EQk=psi(EQ#currentnode);
            NEQk=NEQ#currentnode;
            p=(phi(varno))(EQk);
            for i to numgens(p)-1 do(
                if leadMonomial(p_i)!=1 then(
                    m=redCoeff(leadCoefficient(p_i),NEQk);
                    if m!=0 then(
                        degm=for i from 1 to 2*n list degree(z_i,lift(m,F));
                        if degm!=for i to 2*n-1 list 0 then(
                            J=(gens radical ideal promote(m,R))_(0,0);
		            found=varno;
                            break;
	                );
		    );
	        );
                if found>0 then break;
            );
            varno=varno+1;	
        );		
        if found >0 and found < 2*n then(
            I=GB(radical((phi(found))(EQk+ideal(J))));
	    NEQk=unique(for i to #NEQk-1 list (
                 gens radical saturate(ideal(NEQk#i),I))_(0,0));
	    NEQkk=for i to #NEQk-1 list (NEQk#i)%I;
            if member(0,NEQkk) == false then(
	        EQ=append(EQ,I);
                NEQ=append(NEQ,NEQk);
	        PREV=append(PREV,currentnode); 
                nextnode=nextnode+1;
	    );
            I=GB(radical(saturate((phi(found))(EQk),ideal(J))));
	    NEQk=NEQk|{J};
	    NEQk=unique(for i to #NEQk-1 list (
                 gens radical saturate(ideal(NEQk#i),I))_(0,0));
	    NEQkk=for i to #NEQk-1 list (NEQk#i)%I;
            if member(0,NEQkk) ==false then(
                EQ=append(EQ,I);
                NEQ=append(NEQ,NEQk);
	        PREV=append(PREV,currentnode);
                nextnode=nextnode+1;
	    );
        );
        currentnode=currentnode+1;
        sizeEQ=#EQ;
    );
    (EQ,NEQ,PREV)
);
----------------------------------------------------
multihomPrint:=(V,R,n,leafs)->(
    eq=V#0;
    neq=V#1;
    prev=V#2;
    for j to n-1 do(neq=for i to #neq-1 list 
        delete(promote(-z_(2*j+1)+1,R),neq#i));
    for i to #(eq)-1 do if eq#i!=1 then if (
        leafs==false or member(i,prev)==false) 
        then print(prev#i,i,toString(eq#i),toString(neq#i))
    )  
----------------------------------------------------
--Example 4 above-----------------------------------
R=multihomRing(3,QQ);
V=multihomVariety(3,R,
        ideal(y_2*(y_6^2*y_4+y_6*y_5*y_3+y_5^2*y_3),
              y_6*(y_6^2*y_4+y_6*y_5*y_3+y_5^2*y_3)));
multihomPrint(V,R,3,true)
\end{verbatim}

\begin{verbatim}
(0, 2, 6, ideal(z_1,z_2-1,
                z_3,y_4-1,
                y_5-1,y_6), 
           {})
(0, 1, 3, 8, ideal(z_1-1,z_2,
                   z_3,y_4-1,
                   y_5-1,y_6), 
              {})
(0, 1, 4, 10, ideal(z_1-1,
                    z_3,y_4-1,
                    y_5-1,y_6), 
              {z_2})
(0, 2, 5, 11, ideal(z_1,z_2-1,
                    z_3-1,z_4,
                    y_5^2-y_5,y_6+2*y_5-1), 
               {})
(0, 2, 5, 12, ideal(z_1,z_2-1,
                    z_3-1,
                    y_5-1,z_4*y_6^2+y_6+1), 
               {z_4})
(0, 1, 3 ,7, 14, ideal(z_1-1,z_2,
                       z_3-1,
                       y_5-1,z_4*y_6^3+y_6^2+y_6), 
                 {z_4})
(0, 1, 4 ,9, 15, ideal(z_1-1,
                       z_3-1,z_4,
                       y_5^2-y_5,y_6+2*y_5-1), 
                 {z_2})
(0, 1, 4, 9, 16, ideal(z_1-1,
                       z_3-1,
                       y_5-1,z_4*y_6^2+y_6+1), 
                 {z_2, z_4})
(0, 1, 3, 7, 13, 17, ideal(z_1-1,z_2,
                           z_3-1,z_4,
                           z_5-1,y_6^2+y_6), 
                  {})
(0, 1, 3, 7 ,13, 18, ideal(z_1-1,z_2,
                           z_3-1,z_4,
                           z_5,y_6-1), 
                  {})
\end{verbatim}

So from node 6,
$(\overline{y}_1:\overline{y}_2)=(0:1)$,
$(\overline{y}_3:\overline{y}_4)=(0:1)$,
$(\overline{y}_5:\overline{y}_6)=(1:0)$.

From node 8,
$(\overline{y}_2:\overline{y}_1)=(0:1)$,
$(\overline{y}_4:\overline{y}_3)=(1:0)$,
$(\overline{y}_6:\overline{y}_5)=(0:1)$.

From node 10,
$(\overline{y}_2:\overline{y}_1)=(\overline{y}_2:1)\ :\ \overline{y}_2\neq 0$,
$(\overline{y}_4:\overline{y}_3)=(1:0)$,
$(\overline{y}_6:\overline{y}_5)=(0:1)$.

From node 11,
$(\overline{y}_2:\overline{y}_1)=(1:0)$,
$(\overline{y}_4:\overline{y}_3)=(0:1)$,
$(\overline{y}_6:\overline{y}_5)\ :\
\overline{y}_5(\overline{y}_5-1)=0=\overline{y}_6+2\overline{y}_5-1$.

From node 12,
$(\overline{y}_2:\overline{y}_1)=(1:0)$,
$(\overline{y}_4:\overline{y}_3)=(\overline{y}_4:1)\ :\
\overline{y}_4\neq 0$,
$(\overline{y}_6:\overline{y}_5)=(\overline{y}_6:1)\ :\
\overline{y}_4\overline{y}_6^2+\overline{y}_6+1)=0$.

From node 14,
$(\overline{y}_2:\overline{y}_1)=(0:1)$,
$(\overline{y}_4:\overline{y}_3)=(\overline{y}_4:1)\ :\
\overline{y}_4\neq 0$,
$(\overline{y}_6:\overline{y}_5)=(\overline{y}_6:1)\ :\
\overline{y}_4\overline{y}_6^3+\overline{y}_6^2+\overline{y}_6)=0$.

From node 15,
$(\overline{y}_2:\overline{y}_1)=(\overline{y}_2:1)\ :\
\overline{y}_2\neq 0$,
$(\overline{y}_4:\overline{y}_3)=(0:1)$,
$(\overline{y}_6:\overline{y}_5)\ :\
\overline{y}_5(\overline{y}_5-1)=0=\overline{y}_6+2\overline{y}_5-1$.

From node 16,
$(\overline{y}_2:\overline{y}_1)=(\overline{y}_2:1)\ :\
\overline{y}_2\neq 0$,
$(\overline{y}_4:\overline{y}_3)=(\overline{y}_4:1)\ :\
\overline{y}_2,\overline{y}_4\neq 0$,
$(\overline{y}_6:\overline{y}_5)=(\overline{y}_6:1)\ :\
\overline{y}_4\overline{y}_6^2+\overline{y}_6+1=0$.

From node 17,
$(\overline{y}_2:\overline{y}_1)=(0:1)$,
$(\overline{y}_4:\overline{y}_3)=(0:1)$,
$(\overline{y}_6:\overline{y}_5)=(\overline{y}_6:1)\ :\
\overline{y}_6^2+\overline{y}_6=0$.

From node 18,
$(\overline{y}_2:\overline{y}_1)=(0:1)$,
$(\overline{y}_4:\overline{y}_3)=(0:1)$,
$(\overline{y}_6:\overline{y}_5)=(1:0)$.

\end{document}